\documentclass[12pt,a4paper]{article}
\usepackage{amsthm,amsmath,amssymb}
\usepackage{enumerate}
\usepackage[margin=3cm,a4paper]{geometry}
\usepackage{hyperref}
\hypersetup{
 pdftitle={Rank-Width of Random Graphs},
 pdfauthor={Choongbum Lee, Joonkyung Lee, Sang-il Oum}
}
\newtheorem{theorem}{Theorem}[section]
\newtheorem{corollary}[theorem]{Corollary}
\newtheorem{THM}{Theorem}[section]
\newtheorem{LEM}[THM]{Lemma}

\newtheorem{COR}[THM]{Corollary}
\newtheorem{PROP}[THM]{Proposition}
\theoremstyle{definition}

\newtheorem*{REM}{Remark}
\newcommand{\row}{\operatorname{row}}

\begin{document}
\title{Rank-width of Random Graphs\footnote{
  The first author was supported in part by Samsung Scholarship.
  The second and third authors were supported by SRC Program of Korea Science and Engineering Foundation(KOSEF) 
  grant funded by the Korea government (MOST)(No. R11-2007-035-01002-0).
  The third author was also partially supported by TJ Park Junior Faculty Fellowship.}
}
\author{
  Choongbum Lee\thanks{\texttt{choongbum.lee@gmail.com}} \\
    Department of Mathematics\\ UCLA, Los
Angeles, CA, 90095. 
\and
Joonkyung Lee\thanks{\texttt{jk87@kaist.ac.kr}} 
\\
 Department of Mathematical Sciences\\ KAIST,
 Daejeon 305-701, Republic of Korea.
 \and
 Sang-il Oum\thanks{\texttt{sangil@kaist.edu}}%
 \\
 Department of Mathematical Sciences\\ KAIST,
 Daejeon 305-701, Republic of Korea.
}
\date{January 4, 2010}
\maketitle
\begin{abstract}
 \emph{Rank-width} of a graph $G$, denoted by $\mathbf{rw}(G)$, is a width parameter of graphs introduced by Oum and Seymour (2006).
  We investigate the asymptotic behavior of rank-width of a random graph $G(n,p)$.
We show that, asymptotically almost surely,
(i) if $p\in(0,1)$ is a constant, then $\mathbf{rw}(G(n,p)) = \lceil \frac{n}{3}\rceil-O(1)$,
(ii) if $\frac{1}{n}\ll p \le \frac{1}{2}$, then $\mathbf{rw}(G(n,p)) = \lceil \frac{n}{3}\rceil-o(n)$, 
(iii) if $p = c/n$ and $c > 1$, then  $\mathbf{rw}(G(n,p)) \ge r n$ for some $r = r(c)$, and 
(iv) if $p \le c/n$ and $c<1$, then $\mathbf{rw}(G(n,p)) \le 2$.
 As a corollary, we deduce that $G(n,p)$ has linear tree-width
 whenever $p=c/n$ for each $c>1$, answering a question of Gao
 (2006).

\end{abstract}

\section{Introduction}

\emph{Rank-width} of a graph $G$, denoted by $\mathbf{rw}(G)$, is a graph width parameter introduced by Oum and Seymour~\cite{OS2004} and measures the complexity 
of decomposing $G$ into a tree-like structure. The precise definition will be given in the following section. One fascinating aspect of this parameter lies
in its computational applications, namely, if a class of graphs has bounded rank-width, then many NP-hard problems are solvable on this class in polynomial time; for example, see~\cite{CMR2000}.

We consider the Erd\H{o}s-R\'enyi random graph $G(n,p)$. In this model, a graph $G(n,p)$ on a vertex set $\{1,2,\cdots,n\}$ is chosen randomly as follows: for each unordered pair of vertices, they are adjacent with probability $p$ independently at random. Given a
graph property $\mathcal{P}$, we say that $G(n,p)$ possesses $\mathcal{P}$ \emph{asymptotically almost surely}, or a.a.s.\ for brevity, if the probability that $G(n,p)$ possesses $\mathcal{P}$ converges to 1 as $n$ goes to infinity. A function $f:\mathbb{N}\rightarrow[0,1]$ is called the \emph{sharp threshold} of $G(n,p)$ with respect to having $\mathcal{P}$ if the following hold: if $p\ge cf(n)$ for a constant $c>1$, then $G(n,p)$ a.a.s.\ satisfies $\mathcal{P}$ and otherwise if $p\le cf(n)$ and $c<1$, then $G(n,p)$ a.a.s.\ does not satisfy $\mathcal{P}$.

The following is our main result.
\begin{theorem}\label{main} For a random graph $G(n,p)$, the following holds asymptotically almost surely:
	\begin{enumerate}[(i)]
	  \setlength{\itemsep}{1pt}
  \setlength{\parskip}{0pt}
  \setlength{\parsep}{0pt}
		\item if $p\in(0,1)$ is a constant, then $\mathbf{rw}(G(n,p)) = \lceil \frac{n}{3}\rceil-O(1)$,
		\item if $\frac{1}{n}\ll p \le \frac{1}{2}$, then $\mathbf{rw}(G(n,p)) = \lceil \frac{n}{3}\rceil-o(n)$,
		\item if $p = {c}/{n}$ and $c > 1$, then  $\mathbf{rw}(G(n,p)) \ge r n$ for some $r = r(c)$, and
		\item if $p \le {c}/{n}$ and $c<1$, then $\mathbf{rw}(G(n,p)) \le 2$.
	\end{enumerate}
\end{theorem}

Since $\mathbf{rw}(G)\le\lceil\frac{|V(G)|}{3}\rceil$ for every graph $G$, (i) and (ii) 
of this theorem give a narrow range of rank-width. Note that this theorem also gives a bound when $p\ge\frac{1}{2}$, since the rank-width of $G(n,p)$ in this range can be obtained from the inequality $\mathbf{rw}(\overline{G})\le\mathbf{rw}(G)+1$.

Clique-width of a graph $G$, denoted by $\mathbf{cw}(G)$,
is a width parameter introduced by Courcelle and
Olariu~\cite{CO2000}. It is strongly related to rank-width by the
following inequality by Oum and Seymour~\cite{OS2004}.
\begin{equation}\label{eq:cwrw}
\mathbf{rw}(G)\le \mathbf{cw}(G)\le 2^{ \mathbf{rw}(G)+1} -1.
\end{equation}

Tree-width, introduced by Robertson and Seymour~\cite{RS1984}, is a width parameter measuring how similar a graph is to a tree and is closely related to rank-width. We will denote the tree-width of a graph $G$ as $\mathbf{tw}(G)$. The following inequality was proved by Oum~\cite{Oum2006c}: for every graph $G$, we have
\begin{equation}\label{tw_rw_ineq}
\mathbf{rw}(G)\le\mathbf{tw}(G)+1.
\end{equation}

There have been works on tree-width of random graphs.
Kloks~\cite{Kloks1994} proved that $G(n,p)$ with $p=c/n$ has linear
tree-width whenever 
$c> 2.36$. Gao~\cite{Gao2006} improved this constant to $2.162$ and
even conjectured that $c$ can be improved to a constant less than~$2$.
We improve the above constant to the best possible number, $1$,
by the 
following corollary,  stating that there is the sharp threshold
$p={1}/{n}$ of $G(n,p)$ with respect to having linear tree-width.
\begin{corollary}
  Let $c$ be a constant and let $G=G(n,p)$ with $p=c/n$.
  Then the following holds
  asymptotically almost surely:
  \begin{enumerate}[(i)]
  \item If $c>1$, then 
    rank-width, clique-width,and tree-width of $G$ are  at least $c'n$
    for some constant $c'$ depending only on $c$.
  \item If $c<1$, then
    rank-width and tree-width of $G$ are at most $2$ and clique-width of $G$ is
    at most $5$.
  \end{enumerate}
\end{corollary}
\begin{proof}
  (i) follows Theorem \ref{main} with \eqref{eq:cwrw} and
  \eqref{tw_rw_ineq}. 
  (ii) follows easily due to the
  theorem by Erd\H{o}s and R\'enyi~\cite{ER1960,ER1961} stating that asympototically almost
  surely, each component of  $G(n,p)$ with $p=c/n$, $c<1$
  has at most one cycle. It is straightforward to see that such graphs
  have small tree-width, clique-width, and rank-width.
\end{proof}

\section{Preliminaries}

All graphs in this paper have neither loops nor parallel edges. Let $\Delta(G), \delta(G)$ be the maximum
degree and the minimum degree of a graph $G$ respectively. For two subsets $X$ and $Y$ of $V(G)$, let $E_G(X,Y)$ be the set of ordered pairs $(x,y)$ of adjacent vertices $x \in X$ and $y \in Y$. Let $e_G(X,Y) = |E_G(X,Y)|$. We will omit subscripts if it is not ambiguous.

Let $\mathbb{F}_2=\{0,1\}$ be the binary field. For disjoint subsets $V_1$ and $V_2$ of $V(G)$, let $N_{V_1,V_2}$ be a $0$-$1$ $|V_1|\times |V_2|$ matrix over $\mathbb{F}_2$ whose rows are labeled by $V_1$ and columns labeled by $V_2$, and the entry $(v_1, v_2)$ is $1$ if and only if $v_1 \in V_1$ and $v_2 \in V_2$ are adjacent. We define the \emph{cutrank} of $V_1$ and $V_2$, denoted by $\rho_G(V_1,V_2)$, to be $\mathbf{rank}(N_{V_1,V_2})$.

A tree $T$ is said to be \emph{subcubic} if every vertex has degree 1 or 3. A \emph{rank-decomposition} of a graph $G$ is a pair $(T,L)$ of a subcubic tree $T$ and a bijection $L$ from $V(G)$ to the set of all leaves of $T$. Notice that deleting an edge $uv$ of $T$ creates two components $C_u$ and $C_v$ containing $u$ and $v$ respectively. Let $A_{uv}=L^{-1}(C_u)$ and $B_{uv}=L^{-1}(C_v)$. Under these notations, \emph{rank-width} of a graph $G$, denoted by $\mathbf{rw}(G)$, is defined as 
\[
\mathbf{rw}(G)=\min_{(T,L)}\max_{uv\in E(T)}\rho_G (A_{uv},B_{uv}),
\]
where the minimum is taken over all possible rank-decompositions. We assume $\mathbf{rw}(G)=0$ if $|V(G)|\le 1$.

The following lemma will be used later.
\begin{LEM} \label{lemma_seperator} 
  Let $G=(V,E)$ be a graph with at least two vertices.
  If rank-width of $G$ is at most $k$, then there exist two disjoint
  subsets $V_1,V_2$ of $V$ such that
  \[\text{$|V_1|=\left\lceil\frac{n}{2}\right\rceil$,
    $|V_2|= \left\lceil\frac{n}{3}\right\rceil$, and $\rho_G(V_1, V_2) \le k$}. \]
\end{LEM}

\begin{proof}
Let $k=\mathbf{rw}(G)$. Let $(T,L)$ be a rank-decomposition of width $k$. We claim that there is an edge $e$ of $T$ such that $T\setminus e$ gives a partition $(A,B)$ of $V(G)$ satisfying $|A|\ge{n}/{3}$, $|B|\ge{n}/{3}$ and $\rho_G(A,B)\le k$. Assume the contrary. Then for each edge $e$ in $T$, $T\setminus e$ has a component $C_e$ of $T\setminus e$ containing less than ${n}/{3}$ leaves of $T$. Direct each edge $e=uv$ from $u$ to $v$ if $C_e$ contains $u$. Since this directed tree is acyclic, there is a vertex $t$ in $V(T)$ such that every edge incident with $t$ is directed toward $t$. Then there are at most $3$ components in $T\setminus t$ and each component has less than ${n}/{3}$ leaves of $T$, a contradiction. This proves the claim.

Given sets $A, B$ as above, we may assume $|A|\ge {n}/{2}$. Take $V_1\subseteq A$ and $V_2\subseteq B$ of size $\lceil\frac{n}{2}\rceil$ and $\lceil\frac{n}{3}\rceil$, respectively. Then $\rho_G(V_1,V_2) \le \rho_G(A, B) \le k$.
\end{proof}

\section{Rank-width of dense random graphs}

In this section we will show that if $\frac{1}{n} \ll \min(p,1-p)$, then the rank-width of $G(n,p)$ is a.a.s.\ $\lceil\frac{n}{3}\rceil -o(n)$. Moreover, for a constant $p\in(0,1)$, rank-width of $G(n,p)$ is a.a.s.\ $\lceil\frac{n}{3}\rceil -O(1)$.
This bound is achieved by investigating the rank of random matrices. The following proposition provides
an exponential upper bound to the probability of a random vector falling into a fixed subspace.

\begin{PROP} \label{prop_rvprob}
For $0 < p < 1$, let $\eta = \max (p,1-p)$. Let $v \in \mathbb{F}_2^n$ be a random $0$-$1$ vector whose entries are
1 or 0 with probability $p$ and $1-p$ respectively. Then for each $k$-dimensional subspace $U$ of $\mathbb{F}_2^n$,
\[  \mathbf{P}(v\in U) \le \eta^{n-k}  \]
\end{PROP}
\begin{proof}
 Let $B$ be a $k\times n$ matrix whose row vectors form a basis of $U$. By permuting the columns if necessary, we may assume that the first k columns are linearly independent. For a vector $v \in \mathbb{F}_2^n$, let $v^{(k)}$ be the first $k$ entries of $v$, and note that
 \begin{equation}\label{rvprob}
 \mathbf{P} (v\in U)=\sum_{w\in\mathbb{F}_2^k}\mathbf{P}(v\in U|v^{(k)}=w)\mathbf{P}(v^{(k)}=w).
 \end{equation}
Let $u_1,u_2,\cdots,u_k$ be the row vectors of B. Observe that $\{u_j^{(k)}\}_{j=1}^k$ is a basis of $\mathbb{F}_2^k$. Thus, given $v^{(k)}=w=\sum_{i=1}^k c_iu_i^{(k)}$, we have $v\in U$ if and only if $v=\sum_{i=1}^k c_iu_i$. This implies that given each first $k$ entries of $v$, there is a unique choice of remaining entries yielding $v\in U$. Thus for every $w\in\mathbb{F}_2^k$, $\mathbf{P}(v\in U|v^{(k)}=w)\le\eta^{n-k}$. Combining with \eqref{rvprob}, we obtain
\[
 \mathbf{P}(v\in U) \le \eta^{n-k}\sum_{w\in\mathbb{F}_2^k}\mathbf{P}(v^{(k)}=w) = \eta^{n-k}, 
\]
 and this concludes the proof.
\end{proof}

Let $M(k_1,k_2;p)$ be a random $k_1\times k_2$ matrix whose entries are mutually independent and take value $0$ or $1$ with probability $1-p$ and $p$ respectively. Using Proposition \ref{prop_rvprob}, we can bound the probability that the rank of $M(\lceil\frac{n}{3}\rceil,\lceil\frac{n}{2}\rceil;p)$ deviates from $\lceil\frac{n}{3}\rceil$.

\begin{LEM} \label{lemma_rmprob}
 For $0 < p < 1$, let $\eta = \max(p,1-p)$. Then for every $C > 0$,
 $$
 \mathbf{P}\left(\mathbf{rank}\left(M\left(\Bigl\lceil\frac{n}{3}\Bigr\rceil,\Bigl\lceil\frac{n}{2}\Bigr\rceil;p\right)\right) \le \Bigl\lceil\frac{n}{3}\Bigr\rceil -\frac{C}{\log_2\frac{1}{\eta}} \right) < 2^{(\frac{1}{2}-\frac{1}{6}C) n}.
 $$
\end{LEM}
\begin{proof}
Let $M=M(\lceil\frac{n}{3}\rceil,\lceil\frac{n}{2}\rceil;p)$, $\alpha = \lceil\frac{C}{\log_2 \frac{1}{\eta}}\rceil$, and $\row(M)$ be the linear space spanned by the rows of $M$. We may assume $\lceil\frac{n}{3}\rceil-\alpha\ge 0$. Denote row vectors of $M$ by $v_1,v_2,\cdots,v_{\lceil\frac{n}{3}\rceil}$. Note that $\mathbf{rank}(M)$ is at most $\lceil\frac{n}{3}\rceil -\alpha$ if and only if there are $\lceil\frac{n}{3}\rceil-\alpha$ rows of $M$ spanning $\row(M)$. Thus
$$
\mathbf{P}\left(\mathbf{rank}(M)\le\Bigl\lceil\frac{n}{3}\Bigr\rceil -\alpha\right)\le
\sum_{I}\mathbf{P}\left(\{v_i\}_{i\in I}\mbox{ spans } \row(M)\right)
$$ 
where the sum is taken over all $I\subseteq\{1,2,\cdots,\lceil\frac{n}{3}\rceil\}$ with cardinality $\lceil\frac{n}{3}\rceil -\alpha$.
Let $U_I$ be the vector space spanned by row vectors $\{v_i\}_{i\in I}$. By Proposition \ref{prop_rvprob}, we get
$$
\mathbf{P}\left(\{v_i\}_{i\in I}\mbox{ spans } \row(M)\right)=
\mathbf{P}(\{v_j:j\notin I\}\subseteq U_I)\le(\eta^{\lceil\frac{n}{2}\rceil-\lceil\frac{n}{3}\rceil+\alpha})^{\alpha},
$$
since rows are mutually independent random vectors. Combining these inequalities, we conclude that
$$
\mathbf{P}\left(\mathbf{rank}(M)\le\Bigl\lceil\frac{n}{3}\Bigr\rceil -\alpha\right)\le 2^{\lceil\frac{n}{2}\rceil-1}(\eta^{\alpha})^{\lceil\frac{n}{2}\rceil-\lceil\frac{n}{3}\rceil+\alpha}\le 2^{\frac{n}{2}}2^{-\frac{n}{6}C} = 2^{(\frac{1}{2}-\frac{1}{6}C) n}
$$
because $\lceil\frac{n}{2}\rceil-\lceil\frac{n}{3}\rceil+\alpha\ge\frac{n}{6}$ and $\lceil\frac{n}{2}\rceil\choose{k}$ $\le 2^{\lceil\frac{n}{2}\rceil-1}$.
\end{proof}

\begin{PROP} \label{thm_rwasymptotic}
Let $\eta = \max(p,1-p)$ and $n\ge 2$. Then
$$
\mathbf{P}\left(\mathbf{rw} (G(n,p)) \le \Bigl\lceil\frac{n}{3}\Bigr\rceil - \frac{12.6}{\log_2 \frac{1}{\eta}}\right)<2^{-0.015n}.
$$
\end{PROP}
\begin{proof}
Let $G=G(n,p)$, $\mathcal{S}=\{N_{V_1 ,V_2}:|V_1|=\lceil\frac{n}{2}\rceil, |V_2|=\lceil\frac{n}{3}\rceil\ \mbox{ for disjoint $V_1 ,V_2\subseteq V(G)$} \}$ and let $\mu = \min_{N\in\mathcal{S}} \mathbf{rank}(N).$ By Lemma \ref{lemma_seperator}, we have
$\mu\le\mathbf{rw}(G)$. Thus it suffices to show that
$$
\mathbf{P}\left(\mu\le\Bigl\lceil\frac{n}{3}\Bigr\rceil-\frac{12.6}{\log_2 \frac{1}{\eta}}\right)<2^{-0.015n}.
$$
For each $N\in\mathcal{S}$, let $A_N$ be the event that $\mathbf{rank}(N)\le\lceil\frac{n}{3}\rceil-\frac{12.6}{\log_2 \frac{1}{\eta}}$.
Note that
$$
\mathbf{P}\left(\mu\le\Bigl\lceil\frac{n}{3}\Bigr\rceil-\frac{12.6}{\log_2 \frac{1}{\eta}}\right)=\mathbf{P}(\bigcup_{N\in\mathcal{S}}A_N)\le\sum_{N\in\mathcal{S}}\mathbf{P}(A_N).
$$
By Lemma \ref{lemma_rmprob}, we have $\mathbf{P}(A_N)\le2^{-1.6n}$. Notice also that $|\mathcal{S}|\le 3^n$. Therefore,
$$
\mathbf{P}\left(\mu\le\Bigl\lceil\frac{n}{3}\Bigr\rceil-\frac{12.6}{\log_2 \frac{1}{\eta}} \right)
\le 3^n 2^{-1.6n} <  2^{-0.015n}. \qedhere
$$
\end{proof}
The main theorem directly follows from this proposition.
\begin{THM}
Asymptotically almost surely, $G=G(n,p)$ satisfies the following:
\begin{enumerate}[(i)]
\setlength{\itemsep}{1pt}
  \setlength{\parskip}{0pt}
  \setlength{\parsep}{0pt}
\item if $p\in(0,1)$ is a constant, then $\lceil\frac{n}{3}\rceil - O(1)\le \mathbf{rw} (G) \le \lceil\frac{n}{3}\rceil$, and
\item if $\frac{1}{n}\ll\min(p,1-p)$, then $\lceil\frac{n}{3}\rceil - o(n)\le\mathbf{rw} (G) \le \lceil\frac{n}{3}\rceil$.
\end{enumerate}
\end{THM}

\section{Rank-width of sparse random graphs}

In this section we investigate the rank-width of $G(n,p)$ when $p = {c}/{n}$ for some constant $c > 0$. 
Note that Proposition \ref{thm_rwasymptotic} does not give any information when $p = {c}/{n}$ and $c$ is close to 1. As mentioned in the introduction, the linear lower bound of rank-width in this range of $p$ is closely related to a sharp threshold with respect to having linear tree-width.
  We show that, when $p={c}/{n}$,
\begin{enumerate}[(i)]
\item if $c < 1$, then rank-width is a.a.s. at most $2$,  
\item if $c = 1$, then rank-width is a.a.s.\ at most $O(n^{\frac{2}{3}})$ and,
\item if $c > 1$, then there exists $r =r(c)$ such that rank-width is a.a.s.\ at least $r n$.
\end{enumerate}
Erd\H{o}s and R\'{e}nyi~\cite{ER1960,ER1961} proved that if $c < 1$ then $G(n,p)$ a.a.s.\ consists of trees and unicyclic (at most one edge added to a tree) components and if $c =1$ then the largest component has size at most $O(n^{\frac{2}{3}})$. Therefore, (i) and (ii) follow easily because trees and unicyclic graphs have rank-width at most 2.

Thus, (iii) is the only interesting case. When $c>1$, $G(n,p)$ has a unique component of linear size, called the \emph{giant component}. Hence, in order to prove a lower bound on the rank-width of $G(n,p)$, it is enough to find a lower bound of the rank-width of the giant component.

We need some definitions to describe necessary structures.
Let
$G=(V,E)$ be a connected graph. For a non-empty proper subset $S$ of
$V(G)$, let $d_G(S)= \sum_{v \in S} \deg_G(v)$.
The \emph{(edgewise) Cheeger constant} of a connected graph $G$ is
\[ \Phi(G) = \min_{\emptyset\neq S \subsetneq V(G)} \frac{e_G(S, V(G) \setminus S)}{\min(d_G(S), d_G(V(G) \setminus S))}. \]
\begin{REM}
 In \cite{BKW2006}, the following alternative definition of the
Cheeger constant of a connected graph $G$ is used. For a vertex
$v$, let $\pi_v = \frac{\deg_G(v)}{2|E(G)|}$ and for vertices $v$ and $w$ of $G$,
define
\[p_{vw}=
\begin{cases}
  {1}/\deg_G(v) & \text{if $v$ and $w$ are adjacent,}\\
  0 & \text{otherwise.}
\end{cases}
\]
For a subset $S$ of $V(G)$, let $\pi_G(S) = \sum_{v \in S} \pi_v$.
Thus $d_G(S)=2\lvert E(G)\rvert \pi_G(S)$.
In \cite{BKW2006}, the Cheeger constant of a graph $G$ is defined
alternatively as
\[ \min_{0 < \pi_G(S) \le \frac{1}{2}} \frac{1}{\pi_G(S)}\sum_{i \in S, j \notin S} \pi_i p_{ij}.\]
We can easily see that these definitions are equivalent as follows:
\begin{align*}
\Phi(G)
  = \min_{\emptyset \neq S \subsetneq V(G)} \frac{e_G(S, V(G) \setminus
    S)}{\min(d_G(S), d_G(V(G) \setminus S))}
 &= \min_{0 < \pi_G(S) \le \frac{1}{2}} \frac{e_G(S, V(G) \setminus
    S)}{d_G(S)} \\
  &= \min_{0 < \pi_G(S) \le \frac{1}{2}} \frac{1}{\pi_G(S)}\sum_{i \in S, j \notin S} \pi_i p_{ij} ,
\end{align*}
where the second equality follows from the fact that
$\pi_G(S) + \pi_G(V(G) \setminus S) = 1$.
\end{REM}

Benjamini, Kozma and Wormald \cite{BKW2006} proved the following theorem.

\begin{THM}[Benjamini, Kozma and Wormald \cite{BKW2006}] \label{thm_giantstructure}
Let $c>1$ and $p=c/n$. Then there exist $\alpha, \delta > 0$
such that $G(n,p)$ a.a.s.\ contains a connected subgraph $H$ such that $\Phi(H) \ge \alpha$ and $|V(H)| \ge \delta n$.
\end{THM}

\begin{REM}
  The above theorem is a consequence of \cite[Theorem 4.2]{BKW2006}.
  The graph $H$ in Theorem~\ref{thm_giantstructure} is the graph
  $R_N(G)$ in \cite[Theorem 4.2]{BKW2006}, which proves that $R_N(G)$
  is a.a.s.\ an $\alpha$-strong core of $G$. This means that $R_N(G)$
  is a subgraph of $G$ with $\Phi(R_N(G)) \ge \alpha$ by the
  definitions given in Section 2.2 and Section 3 of \cite{BKW2006}.
  The condition $|V(H)| \ge \delta n$ is not explicit in \cite[Theorem
  4.2]{BKW2006}. However this fact follows from \cite[Lemma
  4.7]{BKW2006}, because $R_N(G)$ must have more vertices than its
  kernel $K(R_N(G))$ (the definition of kernel is given in
  \cite[Section 4]{BKW2006}). Note that $\hat{n}$ in \cite[Lemma
  4.7]{BKW2006} satisfies $\hat{n} = \Omega(n)$ by the remark
  following \cite[Lemma 4.1]{BKW2006}. The proof of Theorem 4.2 given
  in \cite[Section 5]{BKW2006} also mentioned this fact explicitly.
\end{REM}

A graph $H$ with the property as in Theorem \ref{thm_giantstructure} is called an \emph{expander graph}. The simple restriction of $\Phi(H)$ being bounded away from $0$ provides a
strikingly rich structure to the graph as in Theorem
\ref{thm_giantstructure}.
Interested readers are referred to the survey paper \cite{HLW2006}.

By using this expander subgraph $H$, we will show that $G(n,p)$ must have large 
rank-width when $p={c}/{n}$ and $c>1$. Before proving this, we need a technical lemma
which allows us to control the maximum degree of a random graph $G(n,p)$.

\begin{LEM} \label{lemma_degreecontrol}
Let $c>1$ be a constant and $p=c/n$. Then for every $\varepsilon > 0$, there exists $M=M(c, \varepsilon)$ such that $G=G(n,p)$ a.a.s.\ has the following property: 
Let $X$ be the collection of vertices which have degree at least $M$. Then the number of edges incident with $X$ is at most $\varepsilon n$.
\end{LEM}
\begin{proof}
Let $V=V(G)$. Let $M$ be a large number satisfying
\begin{equation}\label{chooseM}
\sum_{k=M}^{\infty}k\frac{c^k}{(k-1)!}<\frac{\varepsilon}{2}.
\end{equation}
For each $v \in V$, define a random variable $Y_v = \deg(v)$ if $\deg(v)\ge M$ and $Y_v=0$ otherwise. Then by \eqref{chooseM},
\begin{equation}
  \begin{split}
    \mathbb{E}[Y_v^2]&= \sum_{k=M}^{n-1} k^2 \mathbf{P}(\deg(v) = k)  \\
    &\le \sum_{k=M}^{n-1} k^2\binom{n-1}{k} \left( \frac{c}{n} \right)^k \le  
    \sum_{k=M}^{\infty} k\frac{c^k}{(k-1)!}<\frac{\varepsilon}{2}.
  \end{split}
  \label{eqn_degreebound}
\end{equation}
Since $Y_v\le Y_v^2$, we also have $\mathbb{E}[Y_v]\le\varepsilon/2$. Note that the number of edges incident with $X$ is at most $\sum_{v\in V}Y_v$. Hence, it is enough to prove a.a.s.\ $Y = \sum_{v \in V} Y_v \le \varepsilon n$. Observe that $\mathbb{E}[Y] \le \frac{\varepsilon}{2}n$. Moreover, the variance of $Y$ can be computed as
\begin{equation} \label{eqn_degreevariance} 
  \begin{split}
    \mathbb{E}[(Y - \mathbb{E}[Y])^2] &= \sum_{v \in V} \big( \mathbb{E}[Y_v^2] - \mathbb{E}[Y_v]^2 \big) + \sum_{v \neq w \in V} \big( \mathbb{E}[ Y_v Y_w ] - \mathbb{E}[Y_v] \mathbb{E}[Y_w] \big)   \\
    &\le \varepsilon n  + \sum_{v \neq w \in V} \big( \mathbb{E}[ Y_v Y_w ] - \mathbb{E}[Y_v] \mathbb{E}[Y_w] \big),
 \end{split}
\end{equation}
where for each $v,w \in V, v \neq w$,  
\begin{multline*}
  \mathbb{E}[ Y_v Y_w ] - \mathbb{E}[Y_v]\mathbb{E}[Y_w]\\
  = \sum_{k,l=M}^{n-1} kl \big( \mathbf{P}(\deg(v) = k,  \deg(w) = l) - \mathbf{P}(\deg(v) = k)\mathbf{P}(\deg(w) = l) \big).
\end{multline*}
Let $q_{k}=\mathbf{P}(\deg(v)=k|vw\notin E(G))=\mathbf{P}(\deg(v)=k+1|vw\in E(G))$, for distinct vertices $v,w$ in $G(n,p)$.
Notice that, given either $vw\in E(G)$ or $vw\notin E(G)$, $Y_v$ and $Y_w$ are independent. Thus, we deduce the following:
\begin{align*}
\lefteqn{\mathbb{E}[ Y_v Y_w ] - \mathbb{E}[Y_v]\mathbb{E}[Y_w]}\\
&= \sum_{k,l=M}^{n-1} kl \big( p q_{k-1} q_{l-1} + (1-p) q_{k} q_{l} - (p q_{k-1} + (1-p)q_{k})(p q_{l-1} + (1-p)q_{l})\big)  \\
&\le p\sum_{k,l=M}^{n-1} kl( q_{k-1} q_{l-1} +  q_k q_l)  \\
&\le 2p \sum_{k=M-1}^{n-1} (k+1)q_{k}  \sum_{l=M-1}^{n-1} (l+1)q_{l}
\quad\le \frac{\varepsilon^2}{n}.
\end{align*}

Last inequality follows from \eqref{chooseM}, since similarly as done in \eqref{eqn_degreebound} we get
\[
\sum_{k=M-1}^{n-1} (k+1) q_{k}=\sum_{k=M-1}^{n-1}(k+1)\binom{n-2}{k} \left( \frac{c}{n} \right)^k 
\le\sum_{k=M}^{\infty}k\frac{c^{k-1}}{(k-1)!}< \frac{\varepsilon}{2c}
\]
and $c>1$. Thus, by \eqref{eqn_degreevariance}, we proved that the variance $\sigma^2$ of $Y$ is at most $(1 + \varepsilon)\varepsilon n$. Finally, using Chebyshev's inequality and the fact $\mathbb{E}[Y_v]\le\varepsilon/2$, we show that
\[ \mathbf{P}\big(Y > \varepsilon n\big)\le\mathbf{P}\left(Y \ge \mathbb{E}[Y] + \frac{\varepsilon n}{2} \right) \le \frac{\sigma^2}{\varepsilon^2n^2 /4} \le \frac{1+\varepsilon}{\varepsilon n/4}, \]
which concludes the proof.
\end{proof}
The following lemma will be used in the proof of the main theorem.
\begin{LEM}\label{rank_control}
Let $A$ be a matrix over $\mathbb{F}_2$ with at least $n$ non-zero entries. If each row and column contains at most $M$ non-zero entries, then $\mathbf{rank}(A)\ge \frac{n}{M^2}$.
\end{LEM}
\begin{proof}
We apply induction on $n$. We may assume $n>M^2$. Pick a non-zero row $w$ of $A$. We may assume that the first entry of $w$ is non-zero, by permuting columns if necessary. Now remove all rows $w'$ whose first entry is 1. 
Since the first column has at most $M$ non-zero entries, we remove at most $M$ rows including $w$ itself. 
Hence, we get a submatrix $A'$ with at least $n-M^2$ non-zero entries. By induction hypothesis, 
$$\mathbf{rank}(A')\ge \frac{n-M^2}{M^2}\ge\frac{n}{M^2}-1.$$
By construction, $w$ does not belong to the row-space of $A'$ and therefore
$$\mathbf{rank}(A)\ge\mathbf{rank}(A')+1\ge\frac{n}{M^2}.\qedhere$$
\end{proof}

\begin{THM}
For $c > 1$, let $p = c/n$. Then there exists $r=r(c)$ such that a.a.s.\ $\mathbf{rw}(G(n,p)) \ge rn$.
\end{THM}
\begin{proof}
Denote $G(n,p)$ by $G$. Let $\alpha, \delta$ be constants from 
Theorem \ref{thm_giantstructure}, and $H$ be the expander subgraph also given by Theorem \ref{thm_giantstructure}. 
Let $W = V(H)$ and let $(W_1,W_2)$ 
be an arbitrary partition of $W$ such that $|W_1|, |W_2| \ge |W|/3$. 
Then since $\Phi(H) \ge \alpha$ and $H$ is connected, we have
\[  \alpha \le \frac{e_{H}(W_1, W_2)}{\min(d_H(W_1), d_H(W_2))} 
           \le \frac{e_{H}(W_1, W_2)}{\min(|W_1|, |W_2|)} \le \frac{e_{G}(W_1, W_2)}{|W|/3} . \]
Thus $e_{G} (W_1, W_2) \ge \frac{\alpha\delta}{3}n$. 
By Lemma \ref{lemma_degreecontrol}, there exists $M$ such that the number of edges incident with a vertex of degree at least $M$ is at most $\frac{\alpha\delta}{6}n$. 
Let $W_1' = W_1 \setminus X$ and $W_2' = W_2 \setminus X$. Since $e_{G}(W_1', W_2')\ge \frac{\alpha\delta}{6}n$, $N_{W_1',W_2'}$ has at least $\frac{\alpha\delta}{6}n$ entries with value 1. Moreover, $N_{W_1',W_2'}$ has at most $M$ entries of value 1 in each row and column. Hence, we can use Lemma \ref{rank_control} to obtain
\[ \frac{\alpha\delta}{6M^2} n \le \rho_G(W_1', W_2') \le \rho_G(W_1, W_2). \]
Since $W_1, W_2$ are arbitrary subsets satisfying $|W_1|, |W_2| \ge |W|/3$,  this implies that the 
induced subgraph $G[W]$ has rank-width at least $\frac{\alpha\delta}{6M^2}n$ by Lemma \ref{lemma_seperator}. Therefore, rank-width of $G$ is at least $\frac{\alpha\delta}{6M^2}n$. 
\end{proof}
\begin{COR}
Let $c > 1$ and $p = c/n$. Then there exists $t = t(c)$ such that a.a.s.\ $\mathbf{tw}(G(n,p)) \ge t n$.
\end{COR}

\paragraph{Acknowledgment.}
Part of this work was done during the IPAM Workshop
 \emph{Combinatorics: Methods and Applications in Mathematics and Computer Science}, 2009.
 The first author would like to thank Nick Wormald for the discussion at IPAM Workshop.

\end{document}